\documentclass[reqno,12pt]{article}

\usepackage{amssymb,amsmath}

\usepackage{bbm}
\usepackage{pbox}
\usepackage{mathrsfs}

\usepackage{multicol}
\usepackage{amsthm}

\usepackage[colorlinks = true,
            linkcolor = blue,
            urlcolor  = blue,
            citecolor = blue,
            anchorcolor = blue]{hyperref}

\usepackage[affil-it]{authblk}

\renewcommand{\phi}{\varphi}
\newcommand{\Lie}{\operatorname{Lie}}
\renewcommand{\L}{\mathcal L}
\newcommand{\C}{\mathbb{C}}
\newcommand{\N}{\mathcal{N}}
\newcommand{\M}{\mathcal{M}}
\newcommand{\gl}{\mathfrak{gl}}
\newcommand{\R}{\mathbb{R}}
\newcommand{\Z}{\mathbb{Z}}
\newcommand{\g}{\mathfrak g}
\newcommand{\h}{\mathfrak h}
\newcommand{\tr}{\operatorname{tr}}
\newcommand{\w}{\wedge}
\newcommand{\om}{\omega}
\newcommand{\Om}{\Omega}
\newcommand{\A}{\mathscr A}
\renewcommand{\P}{\mathcal P}
\newcommand{\Hom}{\text{Hom}}
\newcommand{\eps}{\epsilon}
\newcommand{\ad}{\operatorname{ad}}
\newcommand{\Ad}{\operatorname{Ad}}
\newcommand{\su}{\mathfrak{su}}
\newcommand{\G}{\mathcal G}
\renewcommand{\sl}{\mathfrak{sl}}
\renewcommand{\u}{\mathfrak u}
\newcommand{\Res}{\operatorname{Res}}
\newcommand{\diag}{\operatorname{diag}}
\renewcommand{\Re}{\operatorname{Re}}
\renewcommand{\Im}{\operatorname{Im}}
\renewcommand{\O}{\mathcal O}

\newtheorem{lemma}{Lemma}  
\newtheorem{theorem}{Theorem}
\newtheorem*{theorem*}{Theorem}
\newtheorem{corollary}{Corollary}
\newtheorem{definition}{Definition}
\newtheorem{proposition}{Proposition}
\newtheorem*{prop*}{Proposition}
\newtheorem*{remark}{Remark}

\title{A hyperholomorphic line bundle on certain hyperk\"ahler manifolds not admitting an $S^1$-symmetry}
\author{Eric O. Korman \\
\small Department of Mathematics \\
\small The University of Texas at Austin \\
\small 2515 Speedway, RLM 8.100 \\
\small Austin, TX 78712 \\
\normalfont{\url{ekorman@math.utexas.edu}} \\
}

\date{}
\begin{document}
\maketitle 
\begin{abstract}
Generalizing work of Haydys \cite{Haydys:2008} and Hitchin \cite{Hitchin:2012}, we prove the existence of a hyperholomorphic line bundle on certain hyperk\"ahler manifolds that do not necessarily admit an $S^1$ action.  As examples, we consider the moduli space of (non-strongly) parabolic Higgs bundles, the moduli space of solutions to Nahm's equations, and Nakajima quiver varieties.
\end{abstract}

\tableofcontents

\section{Introduction}
A hyperholomorphic line bundle over a hyperk\"ahler manifold is a line bundle with connection whose curvature is of type (1,1) in each complex structure.  Conversely, given an integral 2-form that is of type (1,1) in each complex structure we can find a line bundle with connection of that curvature.  Work of Haydys \cite{Haydys:2008} and Hitchin \cite{Hitchin:2012,Hitchin:2013} has shown the existence of a canonical hyperholomorphic line bundle on a hyperk\"ahler manifold admitting an $S^1$ action that preserves the metric and one complex structure while rotating the other two.  Specifically, they prove

\begin{theorem} \label{thm:11form}
Suppose $(M, g, \om_I, \om_J, \om_K)$ is a hyperk\"ahler manifold with an isometric action of $S^1$ such that
\[
\L_X \om_I = 0, ~~ \L_X \om_J = -\om_K, ~~ \L_X \om_K = \om_J,
\]
or, equivalently,
\[
d\alpha = 0, ~~ d(J\alpha) = \om_J, ~~ d(K\alpha) = \om_K,
\]
where $X$ is the Killing vector field that generates the $S^1$ action and $\alpha = i_X \om_I$.  Then 
\[
\om_I + d d_I^c \mu
\]
is of type (1,1) in every complex structure, where $\mu \in C^\infty(M)$ is the moment map for the $S^1$-action (in symplectic structure $\om_I$).  In particular, if $(M,\om_I)$ is prequantizable then $M$ admits a hyperholomorphic line bundle with the above form as its curvature.
\end{theorem}

This line bundle is used in a correspondence between hyperk\"ahler and quaternion K\"ahler manifolds \cite{Haydys:2008,Hitchin:2012} and also appears in physics in the case that $M$ is the moduli space of Higgs bundles \cite{Neitzke:2011}.

We will generalize theorem \ref{thm:11form} to
\begin{theorem}\label{thm:mainX}
Let $X$ be any vector field (not necessarily Killing) on a hyperk\"ahler manifold $M$ and let $\alpha = i_X \om_I$.  Suppose there are 2-forms $F_1, F_2$ of type (1,1) in each complex structure such that the following equations,
\begin{equation}
\L_X \om_I = 0, ~~ \L_X \om_J = -\om_K - F_2, ~~ \L_X \om_K = \om_J + F_1 \label{Xprops},
\end{equation}
which are equivalent to
\begin{equation}
d\alpha = 0, ~ d(J\alpha) = \om_J + F_1, ~ d(K\alpha) = \om_K + F_2, \label{alphaprops}
\end{equation}
are satisfied.  Then
\[
\om_I - d(I\alpha)
\]
is of type (1,1) in each complex structure.
\end{theorem}
Note that if $X$ comes from an $S^1$-action with moment map $\mu$, then $d_I^c \mu = -I\alpha$.  Just as the typical examples of hyperk\"ahler manifolds with $S^1$ actions satisfying the conditions of theorem \ref{thm:11form} are cotangent bundles, we will see that hyperk\"ahler manifolds satisfying the conditions of \ref{thm:mainX} look like twisted cotangent bundles. 

Despite the odd first impression of the equations (\ref{Xprops}) and (\ref{alphaprops}), they arise quite naturally.  An example of a manifold satisfying the conditions of theorem \ref{thm:11form} is the moduli space of Higgs bundles, with the $S^1$-action given by scaling the Higgs field.  If one instead looks at the moduli space of Higgs bundles over a Riemann surface where the Higgs fields are allowed to have simple poles along a fixed divisor, then the moduli space $\P$ is a holomorphic Poisson manifold \cite{LogMart}.  The symplectic leaves $\M$ of $\P$ are given by fixing the eigenvalues for the residues of the Higgs fields and have a hyperk\"ahler structure \cite{Konno:1993, Nakajima:1996}.  The $S^1$-action on $\P$ clearly does not preserve this foliation but we show that there is a canonical projection map from $T\P$ to $T\M$ under which the vector field generating the $S^1$-action on $\P$ projects to a vector field that satisfies (\ref{Xprops}) on any symplectic leaf.

The conditions of theorem \ref{thm:mainX} also naturally arise as the result of hyperk\"ahler reduction on a hyperk\"ahler manifold $M$ with $G$-basic 1-form $\alpha$ satisfying the conditions of theorem \ref{thm:11form} (e.g. if $M$ has an $S^1$-action).  Then the conditions of theorem \ref{thm:11form} may not descend to the hyperk\"ahler quotient and need to be replaced by the weaker conditions of \ref{thm:mainX}.  This general set-up is discussed in section \ref{sec:reduction}.

Associated to a hyperk\"ahler manifold $M^{4n}$ is its twistor space $Z$, which is a complex manifold of complex dimension $2n+1$ and fibers over $\C P^1$.  There is a one-to-one correspondence between hyperholomorphic line bundles on $M$ and holomorphic line bundles on $Z$ that are trivial on twistor lines.  In the case that $M$ admits an $S^1$ action, Hitchin \cite{Hitchin:2012} gives a C\v ech description of this line bundle over $Z$.  Further, he shows that this line bundle has a meromorphic connection with singularities on the fibers over the north and south poles of $\C P^1$.  We generalize this C\v ech description in the case that $M$ satisfies the conditions of theorem \ref{thm:mainX}.  The main difference is that there is no longer a meromorphic connection, as the curvature (whose (1,1) part represents the Atiyah class of the holomorphic line bundle) of the analogous connection has terms involving the (1,1)-forms $F_1$ and $F_2$.

\subsection*{Acknowledgements}
 The author is very grateful to Andrew Neitzke for suggesting the problem and for numerous helpful discussions.  This work was supported by NSF grant DMS-1148490.

\section{Proof of theorem \ref{thm:mainX}}
We now prove theorem \ref{thm:mainX}.  Our proof is different than the one given in \cite{Hitchin:2012} and rests on the vanishing of the Nijenhuis tensor in each complex structure.  We first note that the equivalence of equations \ref{Xprops} and \ref{alphaprops} follows from Cartan's homotopy formula and the facts that $i_X \om_J = -K\alpha$ and $i_X \om_K = J\alpha$.

Thus suppose we have a 1-form $\alpha$ on a hyperk\"ahler manifold $M$ satisfying (\ref{alphaprops}).  Since $\om_I$ is of type (1,1) in the $I$ complex structure, to show that $\om_I - d(I\alpha)$ is also (1,1) in the $I$ complex structure we must show that $d(I\alpha)$ is, i.e. that $d(I\alpha)(Iv, Iw) = d(I\alpha)(v,w)$ for all vector fields $v$ and $w$.  We have
\begin{gather*}
d(I\alpha) (Iv, Iw) = -Iv \cdot \alpha(w) + Iw \cdot \alpha(v) - \alpha(I[Iv,Iw]).
\end{gather*}
The vanishing of the Nijenhuis tensor gives us
\[
I[Iv, Iw] = I[v,w] - [Iv,w] - [v, Iw],
\]
which implies that
\begin{align*}
d(I\alpha) (Iv, Iw) &= -Iv \cdot \alpha(w) + \alpha([Iv, w]) + Iw\cdot\alpha(v) + \alpha([v, Iw]) - \alpha(I[v,w]) \\
&= d\alpha(w, Iv) - w\cdot \alpha(Iv) + d\alpha(Iw,v) + v\cdot \alpha(Iw) - \alpha(I[v,w]) \\
&= d\alpha(w, Iv) + d\alpha(Iw, v) + d(I\alpha)(v,w) \\
&= d(I\alpha)(v,w),
\end{align*}
since $d\alpha = 0$.

For the complex structure $J$, we compute
\begin{align*}
d(I\alpha)(Jv,Jw) &= Jv \cdot \alpha(Kw) - Jw \cdot \alpha(Kv) - \alpha(I[Jv, Jw]) \\
&= Jv \cdot \alpha(Kw) - Jw \cdot \alpha(Kv) - \alpha(I[v,w] + K[Jv,w] + K[v, Jw]) \\
&= (Jv\cdot \alpha(Kw) - \alpha(K[Jv,w])) + (-Jw \cdot \alpha(Kv) - \alpha(K[v, Jw])) - \alpha(I[v,w]) \\
&= d(K\alpha)(Jv, w) + w \cdot \alpha(KJv) + d(K\alpha)(v, Jw) - v\cdot \alpha(KJw) - \alpha(I[v,w]) \\
&= d(K\alpha)(Jv, w) + d(K\alpha)(v,Jw) + v \cdot \alpha(Iw) - w \cdot \alpha(Iv) - \alpha(I[v,w]) \\
&= d(K\alpha)(Jv,w) + d(K\alpha)(v,Jw) + d(I\alpha)(v,w) \\
&= \om_K(Jv, w) + F_2(Jv,w) + \om_K(v, Jw) + F_2(v,Jw) + d(I\alpha)(v,w) \\
&= \om_K(Jv, w) + \om_K(v, Jw) + d(I\alpha)(v,w),
\end{align*}
where we have used the Nijenhuis identity
\[
[Jv,Jw] = [v,w] + J[Jv,w] + J[v,Jw]
\]
in the second equality and the last equality follows from the fact that $F_2$ is of type (1,1) in $J$.  Thus to show that $\om_I - d(I\alpha)$ is of type (1,1) in $J$, we need to show that
\[
\om_I(Jv,Jw) - \om_K(Jv,w) - \om_K(v,Jw) = \om_I(v,w).
\]
But the left hand side is
\begin{align*}
g(IJv, Jw) - g(KJv,w) - &g(Kv, Jw) \\
&= g(Kv,Jw) + g(Iv,w) - g(Kv,Jw) \\
&= g(Iv,w) \\
&= \om_I(v,w),
\end{align*}
as desired.

Since $\om_I - d(I\alpha)$ is of type (1,1) in complex structures $I$ and $J$, it is also of type (1,1) in complex structure $K$.

\section{Hyperk\"ahler reduction} \label{sec:reduction}
We will now see how the hyperholomorphic 2-form interacts with hyperk\"ahler reduction.  Suppose a hyperk\"ahler manifold $M$ has a 1-form $\alpha$ satisfying
\begin{equation}
d\alpha = 0, ~~ d(J\alpha) = \om_J, ~~ d(K\alpha) = \om_K \label{alphatotalspace}
\end{equation}
as well as a hamiltonian action of a Lie group $G$ that preserves the hyperk\"ahler structure and $\alpha$ (such an example is the case of an $S^1$ action, as in theorem \ref{thm:11form}, that commutes with the action of $G$).  Let
\[
\mu_G = (\mu_I, \mu_J, \mu_K) : M \to \g^* \otimes \R^3
\]
be the moment map and denote by $M//G$ the hyperk\"ahler reduction at 0, i.e. as a manifold $M//G$ is $\mu_G^{-1}(0)/G$ and the hyperk\"ahler structure is induced from that of $M$ (see e.g. \cite{HKLR} for more details). We want to understand when $\alpha$ descends to the hyperk\"ahler quotient $M//G$ and satisfies (\ref{alphaprops}) for some $F_1$ and $F_2$. 

If $M$ has an $S^1$ action with Killing vector field $X$ then a natural compatibility between the $S^1$ and $G$ actions one may want is the equations
\begin{equation}
 X\cdot \mu_I = 0, ~~ X \cdot \mu_J = -\mu_K, ~~ X \cdot \mu_K = \mu_J,  \label{S1equivariance}
 \end{equation}
which say that the moment map is equivariant with respect to the $S^1$-action on $\R^3$ given by rotation about $(1,0,0)$.  

If $\mu$ is a moment map for symplectic form $\om$ and $Y \in \g$, we let $\mu^Y$ denote the function $x \mapsto \mu(x)(Y)$ on $M$.  Then $d\mu^Y = i_{Y^*} \om$ where $Y^*$ is the action vector field on $M$ coming from $Y$.

\begin{proposition} \label{equivmomentmap}
If $\mu_G$ satisfies (\ref{S1equivariance}), which is equivalent to
\begin{equation}
\alpha(Y^*) = 0, ~~ (J\alpha)(Y^*) = -\mu_J^Y, ~~ (K\alpha)(Y^*) = -\mu_K^Y ~ \text{for all $Y \in \g$}, \label{alphamoment}
\end{equation}
then $\alpha$ descends to a 1-form $\hat\alpha$ on the hyperk\"ahler quotient $M//G$, which continues to satisfy (\ref{alphatotalspace}).
\end{proposition}
\begin{proof}
By the previous comments, we have 
\[
X \cdot \mu_I^Y = d\mu_I^Y(X) = \om_I(Y^*, X) = -\alpha(Y^*),
\]
\begin{equation}
X \cdot \mu_J^Y = d\mu_J^Y(X) = \om_J(Y^*, X) = \om_I(X, KY^*) = (K\alpha)(Y^*), \label{XmuJ}
\end{equation}
and, similarly,
\begin{equation}
X \cdot \mu_K^Y = -(J\alpha)(Y^*). \label{XmuK}
\end{equation}
This establishes that (\ref{S1equivariance}) and (\ref{alphamoment}) are equivalent.  But from (\ref{alphamoment}), we see that $\alpha, J\alpha$ and $K\alpha$ are all $\G$-basic when restricted to $\mu_G^{-1}(0)$.  Therefore they descend to forms on $M//G$ that continue to satisfy (\ref{alphaprops}).
\end{proof}

If we only impose that $\alpha$ be $G$-invariant, then equations (\ref{S1equivariance}) only hold up to locally constant functions:
\begin{proposition}\label{commutingactions}
The form $\alpha$ being $G$-invariant is equivalent to any of the following holding for all $Y \in \g$.
\begin{multicols}{2}
\begin{enumerate}
\item $[X, Y^*] = 0$.

\item $\L_{Y^*} \alpha = 0$.

\item $d(X\cdot \mu_I^Y) = 0$.

\item $d(X \cdot \mu_J^Y + \mu_K^Y) = 0$.

\item $d(X\cdot \mu_K^Y - \mu_J^Y) = 0$.

\item $d(\alpha(Y^*)) = 0$.

\item $d( J\alpha(Y^*) + \mu_J^Y) = 0$.

\item $d(K\alpha(Y^*) + \mu_K^Y) = 0$.
\end{enumerate}
\end{multicols}
\end{proposition}
\begin{proof}
Since $\om_I$ is $G$-invariant the $G$-invariance of $\alpha$ is equivalent to $X$ being $G$-invariant, which is equivalent to 1.  We have
\[
\L_{Y^*} \alpha = \L_{Y^*}  i_X \om_I = i_X \L_{Y^*}\om_I + i_{[Y^*,X]} \om_I = i_{[Y^*,X]} \om_I,
\]
from which we get $(1)\Leftrightarrow(2)$ by the non-degeneracy of $\om_I$.  Also,
\[
d(X\cdot \mu_I^Y) = d(\om_I(Y^*,X)) = -d i_{Y^*} \alpha =  -\L_{Y^*}\alpha,
\]
showing that $(2)\Leftrightarrow (3)$.

For (4) we have
\begin{align*}
d(X\cdot \mu_J^Y) &= \L_X d \mu_J^Y \\
&= \L_X i_{Y^*} \om_J \\
&= i_{Y^*} \L_X \om_J + i_{[X,Y^*]} \om_J \\
&= -i_{Y^*} \om_K + i_{[X,Y^*]} \om_J \\
&= -d\mu_K^Y + i_{[X,Y^*]} \om_J.
\end{align*}
Thus
\[
d(X\cdot \mu_J^Y + \mu_K^Y) = i_{[X,Y^*]} \om_J,
\]
which gives $(1)\Leftrightarrow(4)$ by the non-degeneracy of $\om_J$.  A similar calculation shows $(1) \Leftrightarrow (5)$.

Cartan's homotopy formula and the fact that $\alpha$ is closed gives $(2) \Leftrightarrow (6)$.

Finally, equations (\ref{XmuJ}) and (\ref{XmuK}) show that $(4)\Leftrightarrow (8)$ and $(5)\Leftrightarrow (7)$.

\end{proof}

In the examples we will consider (moduli spaces of parabolic Higgs bundles and solutions to Nahm's equations), only the first equation in (\ref{alphamoment}) is satisfied (i.e. $\alpha$ is $G$-basic but not necessarily $J\alpha$ or $K\alpha$).  From (\ref{commutingactions}), the functions $J\alpha(Y^*)$ and $K\alpha(Y^*)$ are locally constant on $\mu_G^{-1}(0)$.  Assuming these are actually constant, we get linear maps
\begin{gather*}
(J\alpha)_\g : \g \to \R, ~~ Y \mapsto (J\alpha)(Y^*) \\
(K\alpha)_\g : \g \to \R, ~~ Y \mapsto (K\alpha)(Y^*).
\end{gather*}
\begin{proposition}
$(J\alpha)_\g$ and $(K\alpha)_\g$ are Lie algebra homomorphisms, i.e. they vanish on $[\g,\g]$.
\end{proposition}
\begin{proof}
Since $\om_J$ is basic and $(J\alpha)(Y_j^*)$ is constant, we have
\begin{gather*}
0 = \om_J(Y_1^*, Y_2^*) = d(J\alpha)(Y_1^*, Y_2^*) = -J\alpha([Y_1^*, Y_2^*]) = -J\alpha([Y_1,Y_2]^*)
\end{gather*}
and similarly for $K\alpha$.
\end{proof}

We will denote the hyperk\"ahler structures on $M//G$ by $\hat{}$, e.g. $\hat I, \hat \om_I$, etc.  Let $\Om \in \A^2(\mu_G^{-1}(0); \g)$ be the curvature of the principal $G$-connection on $\mu_G^{-1}(0) \to M//G$ induced by the metric on $M$, which is of type (1,1) in all complex structures \cite{Gocho-Nakajima:1992}.  Then $(J\alpha)_\g\circ \Om$ and $(K\alpha)_\g\circ \Om$ give characteristic classes that obstruct the equation $d(\hat J\hat \alpha + i \hat K \hat\alpha) = \om_J + i \om_K$:

\begin{proposition} \label{prop:reduction}
Suppose $\alpha$ is $G$-basic, satisfies (\ref{alphatotalspace}), and the functions $J\alpha(Y^*), K\alpha(Y^*)$ are constant for all $Y \in \g$ (when restricted to $\mu_G^{-1}(0)$).  Then $\alpha$ naturally descends to a 1-form $\hat\alpha$ on $M//G$ satisfying the conditions of theorem \ref{thm:mainX}, i.e.
\[
d\hat \alpha = 0, ~ d(\hat J\hat\alpha) = \hat\om_J + F_1, ~ d(\hat K\hat\alpha) = \hat\om_K + F_2,
\]
with $F_1$ and $F_2$ of type (1,1) in each complex structure.  Specifically,
\begin{equation}
F_1 = (J\alpha)_\g \circ \Om, ~~ F_2 = (K\alpha)_\g \circ \Om. \label{Fcurv}
\end{equation}
\end{proposition}
This proposition follows from the following general fact:
\begin{lemma}
Suppose $P\stackrel{\pi}{\to} X$ is a principal $G$ bundle with connection of curvature $\Om$, $\beta \in \A^1(P)$ is $G$-invariant and $d\beta = \pi^*\gamma$ is $G$-basic.  Then for all $Y \in \g$, $\beta(Y^*)$ is locally constant and, assuming this is actually constant, we have
\[
d\hat \beta = \gamma + \beta_\g \circ \Om
\]
where $\hat\beta$ is the 1-form on $X$ coming from the connection and $\beta_\g : \g \to \R$ is the map $Y \mapsto \beta(Y^*)$.
\end{lemma}

Thus it may happen (as we will see in the examples) that while $\alpha$ naturally descends to $\hat\alpha$ on the hyperk\"ahler quotient, the dual vector field $X$, while $G$-invariant, is not tangent to $\mu_G^{-1}(0)$.  If $\hat X$ is the vector field on $M//G$ dual to $\hat\alpha$, then its horizontal lift to $\mu_G^{-1}(0)$ is the orthogonal projection of $X$ onto the level set and the $\R$-action determined by $\hat X$ may not be an $S^1$-action.

Unlike in the case of an $S^1$-action, the K\"ahler forms $\om_J$ and $\om_K$ are no longer exact.  It is thus natural to ask when they are pre-quantizable, i.e. when their cohomology classes live in $H^2(M; 2\pi \Z)$.  By the previous proposition, $\om_J$ and $\om_K$ are cohomologous to $F_1$ and $F_2$, respectively.  However, if the representation $i(J\alpha)_\g \to i\R$ lifts to $G \to U(1)$ then the unitary line bundle $\mu_G^{-1}(0) \times_G \C$ has a connection with curvature $i(J\alpha)_\g \circ \Om = F_1$ (and similarly for $K\alpha$ and $F_2$). Thus we see

\begin{corollary} \label{quantization}
If the representation $i(J\alpha)_\g : \g \to i\R$, (resp. $i(K\alpha)_\g : \g \to i\R$), lifts to $G \to U(1)$ then $\om_J$ (resp. $\om_K$) is prequantizable.
\end{corollary}

\subsection{Push-down}
There is a natural way to pushdown the hyperholomorphic line bundle on $M$ to $M//G$.  We first note that topologically this line bundle is the prequantum line bundle for $\om_I$ and so has the infinitesimal Kostant action of $G$.  The hyperholomorphic connection is $G$-invariant since it is $\nabla = \nabla^{pq} - I\alpha$ with $I\alpha$ $G$-invariant.  Now we restrict it to the level set $\mu^{-1}_G(0)\stackrel{\iota}{\to} M$, which is a principal $G$-bundle over $M//G$ that has a canonical connection given by the metric on $M$.  Then the push-downed line bundle with connection on $M//G$ is defined via
\[
(\hat\nabla, (\iota^* L)/G), ~~ \hat\nabla_v = \nabla_{v^H},
\]
where $v^H$ is a horizontal lift of $v$. Since the pushdown of $\nabla^{pq}$ is the pre-quantum connection $\nabla^{M//G, pq}$ on $M//G$, we have $\hat\nabla = \nabla^{M//G, pq} - \hat I \hat \alpha$.  Therefore the hyperholomorphic line bundle on $M//G$ is obtained via pushdown from the hyperholomorphic line bundle on $M$.  

\section{The line bundle on twistor space} \label{sec:twistor}

\subsection{Twistor space}
Recall that associated to a hyperk\"ahler manifold is its twistor space $Z\stackrel{\pi}{\to} \C P^1$.  As a smooth manifold, $Z = M \times \C P^1$ but the complex structure at $(x,\zeta)$ is $I_\zeta \oplus I_{\C P^1}$, where $I_\zeta = a I + b J + c K$ for $\zeta = (a,b,c) \in S^2 \simeq \C P^1$.   There is a one to one correspondence between hyperholomorphic line bundles on $M$ and holomorphic line bundles on $Z$ that are trivial when restricted to twistor lines.

Let $T_V = \ker \pi_*$ be the vertical vectors and $d_V$ be the vertical de Rham differential on $\A_V^\bullet = \Gamma(Z; \Lambda^\bullet T_V^*)$.  Let $\A^{p,q}_Z(2)$ (resp. $\A^{p,q}_V(2)$) denote the space of $(p,q)$ forms on $Z$ (resp. sections of $\Lambda^{p,q}T_V^*$) with simple singularities on the divisor $Z_0 + Z_\infty = \{\zeta = 0\} \cup \{\zeta = \infty\}$.  $Z$ comes equipped with the following:

\begin{itemize}
\item A real structure, i.e. an anti-holomorphic involution
\[
\tau : Z \to Z, ~~ (x, \zeta) \mapsto \left(x, -\frac{1}{\bar\zeta}\right).
\]

\item A vertical meromorphic symplectic form
\[
\om = \frac{1}{i\zeta}(\om_J + i \om_K) + 2\om_I + \frac{\zeta}{i} (\om_J - i \om_K) \in \A^{2,0}_V(2) 
\]

\end{itemize}

Our construction will use the map 
\begin{equation}
\overline{\tau^*}: \A^\bullet_Z \to \A^\bullet_Z, ~~ \gamma \mapsto \overline{\tau^*\gamma}, \label{realstruct}
\end{equation}
which preserves the type decomposition of differential forms and commutes with $d$.

\subsection{The Lie algebroid}

We can generalize the construction in \cite{Hitchin:2012} of the holomorphic line bundle on $Z$ corresponding to the hyperholomorphic line bundle on a simply-connected hyperk\"ahler manifold $M$ with 1-form $\alpha$ satisfying (\ref{alphaprops}).  Actually, as in \cite{Hitchin:2012}, we will construct a holomorphic Lie algebroid extension
\[
0 \to \O_Z \to E \to T^{1,0} Z \to 0,
\]
isomorphism classes of which correspond to the C\v ech cohomology group $H^1(d\O_Z)$.  Such a Lie algebroid is equivalent to a line bundle if the characteristic class in $H^2(Z;\C)$, which comes from the short exact sequence of sheaves $0 \to \C \to \O_Z \to d\O_Z \to 0$, is integral.

Relative to a cover $U_j$ of $U = \{\zeta \ne \infty\}$ we will construct $\phi_j \in \A^{1,0}(U_j - \{\zeta = 0\})$ that satisfy the following
\begin{enumerate}
\item $d \phi_j = d\phi_k$.
\item $(\zeta \phi_j)\vert_{\zeta = 0} = (\zeta \phi_k)\vert_{\zeta = 0}$.
\item $\overline{\tau^*}(d\phi_j) = - d\phi_j$
\end{enumerate}

From this we see that $\{\phi_k - \phi_j\}$ gives a 1-cocycle of closed holomorphic 1-forms and therefore defines a holomorphic Lie algebroid on $U$.  To extend this to all of $Z$ we observe that, by point 3., the forms $-\overline{\tau^*}(\phi_j)$ give a singular connection on a Lie algebroid over $\{\zeta \ne 0\}$ of the same curvature.  Therefore, the collection $\{\phi_k - \phi_j, \overline{\tau^*}(\phi_k) - \phi_j, \overline{\tau^*}(\phi_j) - \overline{\tau^*}(\phi_k) \}$ gives a 1-cocycle of closed 1-forms on all of $Z$.

When there is an $S^1$-action, the connection $\{\phi_j, -\overline{\tau^*}(\phi_k)\}$ on $Z \backslash (Z_0 \cup Z_\infty)$ is holomorphic, but now the curvature picks up the term $\frac{1}{\zeta} (F_1+iF_2) + \zeta (F_1 - i F_2)$, which is of type (1,1).  Thus the Atiyah class of the line bundle on $Z\backslash (Z_0 \cup Z_\infty)$ is 
\[
\left[\frac{1}{\zeta} (F_1+iF_2) + \zeta (F_1 - i F_2)\right] \in H^{1,1}(Z\backslash (Z_0 \cup Z_\infty)) \simeq H^1(\Om^1_{Z\backslash (Z_0 \cup Z_\infty)}),
\]
which obstructs the existence of a meromorphic connection on the line bundle.

Since our construction follows \cite{Hitchin:2012} very closely and we do not use it in the following examples, the details appear in the appendix.

\section{Examples}
We now focus on three examples: moduli spaces of parabolic Higgs bundles. Nakajima quiver varieties, and moduli spaces of solutions to Nahm's equations.

\subsection{Moduli space of parabolic Higgs bundles} \label{sec:parHiggs}

Following the construction of the moduli spaces of parabolic Higgs bundles of Konno \cite{Konno:1993} and Nakajima \cite{Nakajima:1996}, we fix the following data:
\begin{itemize}

\item A closed Riemann surface $\Sigma$ with a topological vector bundle $E \to \Sigma$ of rank $r$ with trivial determinant bundle.

\item A divisor $D = p_1 + \cdots + p_n$.

\item A flag of $E_{p_j}$ for each $j$ (which we assume to be complete for simplicity).

\item Parabolic weights $\alpha_k^{(j)}, k=1,\ldots, r, j = 1, \ldots, n$, that satisfy
\[
0 \le \alpha_1^{(j)} < \ldots < \alpha_r^{(j)} < 1.
\]

\item Numbers $\lambda_k^{(j)} \in \C, ~~ k=1,\ldots, r, j = 1, \ldots, n$ such that $\sum_k \lambda_k^{(j)} = 0$.  These will be the eigenvalues of the residues of the Higgs fields at the punctures.

\item A singular hermitian metric $h$ that at each puncture $p_j$ takes the form
\[
h = \diag(|z_j|^{2\alpha_1^{(j)}}, \ldots, |z|^{2\alpha_r^{(j)}})
\]
with respect to the flag, where $z_j$ is a local holomorphic coordinate vanishing at $p_j$.

\end{itemize}

\begin{definition}
A parabolic Higgs bundle (with respect to the above data) is a pair $(\bar\partial_E, \theta)$ where $\bar\partial_E$ is a holomorphic structure on $E$ and $\theta \in \Om^1(E; \operatorname{Par} \sl(E) (D))$ is a meromorphic 1-form such that $\Res_{p_j} \theta \in \sl(E_{p_i})$ preserves the parabolic structure and has eigenvalues $\lambda^{(j)}_1, \ldots, \lambda^{(j)}_r$.
\end{definition}

Note that in much of the literature parabolic Higgs bundles refers to the special case where the residues are nilpotent.  We will call such parabolic Higgs bundles strongly parabolic.

As a hyperk\"ahler manifold, the moduli space is obtained via hyperk\"ahler reduction of the infinite dimensional affine space
\[
\mathcal C = \{\text{singular $\mathfrak{su}(E,h)$-connections}\} \times \A^{1,0}_{\underline\lambda}(\Sigma;\operatorname{Par}\mathfrak{sl}(E)(D))
\]
where $\A^{1,0}_{\underline\lambda}(\Sigma;\operatorname{Par}\mathfrak{sl}(E)(D))$ consists of all $\theta \in \A^{1,0}(\Sigma-D; \mathfrak{sl}(E))$ such that 
\begin{itemize}
\item $z_j \theta$ is smooth at each $p_j$, where $z_j$ is a local holomorphic coordinate centered at $p_j$, and $\Res_{p_j}\theta$ preserves the parabolic structure at $p_j$.
\item $\Res_{p_j} \theta$ has eigenvalues $\{\lambda^{(j)}_k\}$.
\end{itemize}

$\mathcal C$ is an affine space modeled on $\A^{0,1}(\Sigma;\mathfrak{sl}(E)) \times \A^{1,0}(\Sigma; \operatorname{SPar}\mathfrak{sl}(E)(D))$, where $\operatorname{SPar}\mathfrak{sl}(E)$ is the space of traceless endomorphisms that are nilpotent at the punctures.  The hyperk\"ahler structure is given by
\begin{gather*}
g((a,b), (a,b)) = 2i\int_\Sigma \tr(a^* \w a + b \w b^*), \\
I(a,b) = (ia, ib), ~~ J(a,b) = (ib^*, -ia^*), ~~ K(a,b) = (-b^*, a^*),
\end{gather*}
for $a \in \A^{0,1}(\Sigma;\mathfrak{sl}(E)), b\in \A^{1,0}(\Sigma; \operatorname{SPar}\mathfrak{sl}(E)(D))$.  

The group $\G = \A^0(\operatorname{Par} SU(E))$ of parabolic special unitary gauge transformations acts on the affine space preserving the hyperk\"ahler structure and Hitchin's equations 
\begin{align}
\label{Hitchinseqs}
\begin{split}
F_A + [\theta,\theta^*] &= 0 \\
\bar\partial_A \theta &= 0
\end{split}
\end{align}
arise as the zero level set of the moment map $\mu_\G$.  The hyperk\"ahler quotient $\M$ is the moduli space of parabolic Higgs bundles.  We note that to rigorously define this space, one must use weighted Sobolev spaces as in \cite{Konno:1993} and \cite{Nakajima:1996}, but we ignore this technical issue.

If all of the $\lambda_k^{(j)}$ are zero (or we consider non-singular Higgs bundles), then there is an $S^1$-action given by multiplication on the Higgs field.  The moment map of this action is $\mathcal C \ni (A, \theta)\mapsto -i\int_\Sigma \tr(\theta\w\theta^*)$ and its exterior derivative is
\begin{equation}
\alpha_{(A,\theta)}(a,b) = -i \int_\Sigma \tr(\theta \w b^* + b\w \theta^*) \label{alphaHiggs}
\end{equation}
In the general case with $\lambda_k^{(j)}$ not all zero, the integral defining the moment map diverges but the integral defining $\alpha$ converges.  This is because near a puncture $p_j$
\begin{gather*}
\theta = \left(\begin{array}{cccc}
   \lambda_1^{(j)} & & \makebox(0,0){\text{\huge 0}} & \\
    & \ddots & & \\
    \makebox(0,0){\text{\huge*}} & & \lambda_r^{(j)}\\
 \end{array}\right) \frac{dz}{z} + \text{higher order terms} \\
b^* = |z|^\gamma\left(\begin{array}{cccc}
   0 & & \makebox(0,0){\text{\huge *}} & \\
    & \ddots & & \\
    \makebox{\text{\huge0}} & & 0\\
 \end{array}\right) \frac{d\bar z}{\bar z} + \text{higher order terms}
\end{gather*}
where $\gamma > 0$.  Then $\tr(\theta \w b^*)$ is of the order $|z|^\gamma \frac{dz \w d\bar z}{|z|^2}$ which in polar coordinates $z = |z|e^{i\theta}$ is $-2i |z|^{1-\gamma} d|z| d\theta$, which is integrable.

It is straightforward to check that $\alpha$ satisfies
\[
d\alpha = 0, ~ d(J\alpha) = \om_J, ~ d(K\alpha) = \om_K, ~ \alpha(Y^*) = 0
\]
for $Y \in \A^0(\operatorname{Par}\mathfrak{su}(E))$ (here the action field is $Y^*_{(A,\theta)} = (\bar\partial_A Y, [\theta, Y]))$.  Thus to invoke proposition \ref{prop:reduction}, we just need to check that on $\mu_\G^{-1}(0)$, the (necessarily locally constant) functions 
\[
(A,\theta)\mapsto (J\alpha)_{(A,\theta)}(Y^*), (K\alpha)_{(A,\theta)}(Y^*)
\]
are indeed constant, which happens if and only if the function $(J\alpha + i K\alpha)(Y^*)$ is constant in $(A,\theta)$.

\begin{proposition}
We have
\[
(J\alpha + i K\alpha)(\bar\partial_A Y, [\theta,Y]) = -2 \sum_{j=1}^n \tr(\diag(\lambda_1^{(j)}, \ldots, \lambda_r^{(j)}) Y_{p_j}),
\]
where the linear map $\diag(\lambda_1^{(j)}, \ldots, \lambda_r^{(j)})$ is represented via the flag at $p_j$.
\end{proposition}
\begin{proof}
A straightforward computation gives
\[
(J\alpha + iK\alpha)(a,b) = 2 \int_\Sigma \tr(\theta \w a).
\]
Thus
\begin{align*}
(J\alpha + iK\alpha)(\bar\partial_A Y, [\theta,Y]) &= 2 \int_\Sigma \tr(\theta \w \bar\partial_A Y) \\
&= -2 \int_\Sigma \tr\left(\bar\partial_A (\theta a)\right) \\
&= -2 \int_\Sigma d \tr(\theta Y) \\
&= -2 \sum_j \tr(\Res \theta_{p_j} Y_{p_j}).
\end{align*}
where the second equality comes from $\bar\partial_A \theta = 0$ and the last line is a consequence of the residue theorem.
\end{proof}

Therefore by proposition \ref{prop:reduction} and theorem \ref{thm:mainX}
\begin{theorem}
The moduli space $\M$ has a hyperholomorphic line bundle (or hyperholomorphic Lie algebroid extension if $\om_I$ is not quantizable) of curvature $2i\om_I - 2id(I\alpha)$.
\end{theorem}
\begin{remark}
In the case of $SU(2)$ Higgs bundles, Konno \cite{Konno:1993b} shows that $\om_I$ is prequantizable if the parabolic weights satisfy
\[
2 \alpha^{(1)}, \ldots, 2 \alpha_r^{(n)}, \sum_{j=1}^n \alpha^{(j)} \in \Z.
\]
where $\alpha^{(j)}$ is the parabolic weight at the $j$th puncture.
\end{remark}

On an open dense set, the space $\M$ is a twisted cotangent bundle over the moduli space $\N$ of parabolic vector bundles.  This subspace of $\M$ is simply-connected since $\N$ is \cite{DasWen:1997} and so we can use the construction in section \ref{sec:twistor} to construct the holomorphic Lie algebroid on the twistor space of the twisted cotangent bundle.

Over a puncture, we have 
\[
\operatorname{Par}\mathfrak{su}(E_{p_j}) \simeq \{(it_1^{(j)}, \ldots, it_r^{(j)}) \in \mathfrak u(1)^r \mid t_1^{(j)} + \cdots + t_r^{(j)} = 0\}.
\]
and the representation $(J\alpha + i K\alpha)_{\Lie\G}$ is the composition
\[
\A^0(\operatorname{Par}\mathfrak{su}(E)) \to \prod_{j=1}^n \operatorname{Par}\mathfrak{su}(E_{p_j}) \stackrel{\underline \lambda}{\to} \C,
\]
where the last map is the representation of $\prod_{j=1}^n \mathfrak u(1)^r$ that has weights $-2\lambda_1^{(j)}, \ldots, -2\lambda_r^{(j)}$ on the $j$th factor.  If $\lambda_k^{(j)} \in \frac{r}{2} \Z$ then this last representation lifts to the Lie group 
\begin{align*}
\prod_{j=1}^n\operatorname{Par} &SU(E_{p_j}) / \Z_r \\
&\simeq \prod_{j=1}^n \{(e^{2\pi it_1^{(j)}}, \ldots, e^{2\pi it_r^{(j)}}) \in U(1)^r \mid t_1^{(j)} + \cdots + t_r^{(j)} = 0\} / \Z_r,
\end{align*}
in which case we also get a lift of the representation of $(J\alpha + i K\alpha)_{\Lie \G}$ to $\G$ (here $\Z_r$ is the group of constant gauge transformations, given by the $r$th roots of unity).  Thus from corollary \ref{quantization} we see

\begin{proposition}
The symplectic form $\om_J$ (resp. $\om_K$) is pre-quantizable if $\Im\lambda_k^{(j)} (\text{resp. $\Re \lambda_k^{(j)}$}) \in \frac{r}{2} \Z$ for all $j,k$.
\end{proposition}

\begin{remark}
From this perspective we can relate the form $F_1 + i F_2$ to hyperholomorphic structures that occur in wall-crossing in physics \cite{Gai-Mo-Neit:2012}.
We have the exact sequence
\[
1 \to \G_{\underline p} \to \G \to \prod_{j=1}^n\operatorname{Par} SU(E_{p_j}) / \Z_r \to 1,
\]
where $\G_{\underline p}$ is the normal subgroup of gauge transformations that restrict to the identity at every puncture.  Then the moduli space $\M$ can be obtained by performing hyperk\"ahler reduction in steps: we can first form the hyperk\"ahler quotient of $\mathcal C$ by the action of $\G_{\underline p}$.  The resulting (finite-dimensional) hyperk\"ahler manifold, $\M'$, will have a hamiltonian action of $\prod_{j=1}^n\operatorname{Par} SU(E_{p_j}) / \Z_r$ and taking the hyperk\"ahler quotient gives $\M$.  If we let $\mu_{\underline p}$ denote this last moment map, then $\mu_{\underline p}^{-1}(0) \to \M$ is a principal $\prod_{j=1}^n\operatorname{Par} SU(E_{p_j}) / \Z_r$-bundle.  The curvature of this principal bundle will be a 2-form with values in $\su(E_{p_j})$ which is (1,1) in each complex structure \cite{Gocho-Nakajima:1992}.  This hyperholomorphic projective bundle is considered in the case of (non-singular) Higgs bundle in \cite{Gai-Mo-Neit:2012}.  Now because of the complete flag we have a decomposition of the fibers $E_{p_j}$ into a direct sum of lines, which gives a decomposition of the curvature 2-form in terms of scalar 2-forms.  Then the 2-form $F_1 + i F_2$ is the linear combination of these 2-forms weighted by (-2 times) the eigenvalues of the residues of the Higgs fields.
\end{remark}

From (\ref{Fcurv}) the form $F_1 + i F_2$ is $(J\alpha + i K\alpha)_{\Lie(\G)} \circ \Om$, where $\Om$ is the curvature of the connection on the infinite rank principal bundle $\G \to \mu_\G^{-1}(0) \to \M$.  Using similar arguments as in \cite{DonKron,GrosPark}, one sees that
\[
\Om(A, \theta)((a_1, b_1),(a_2, b_2)) = -2 G_{(A,\theta)} \ad^*_{(a_1,b_1)}(a_2,b_2),
\]
where $G_{(A,\theta)}$ is the Green's operator in degree 0 of the complex
\[
\A^0(\operatorname{Par} \mathfrak{su}(E)) \stackrel{\nabla^A + \theta}{\to} \A^1(\mathfrak{su}(E)(D)) \to \cdots.
\]

We therefore see that $F_1 + i F_2$ vanishes in the abelian case of $U(1)$-Higgs bundles.  Indeed, in the $U(1)$ case Hitchin's equations decouple and the moduli space is just the product space of holomorphic line bundles with $H^0(K(D))_{\underline\lambda}$, the space of meromorphic 1-forms with simple poles at each $p_j$ of residue $\lambda^{(j)}$.  $H^0(K(D))_{\underline\lambda}$ is affine on $H^0(K)$ and picking a fixed parabolic Higgs field $\theta_0 \in H^0(K(D))_{\underline\lambda}$ determines a diffeomorphism that preserves the hyperk\"ahler structures.  From our set-up, and the discussion to take place in the following section, we see that there is actually a canonical choice of $\theta_0$ given by the vanishing locus of $\alpha$, i.e. $\theta_0$ is determined by
\[
\int_\Sigma \theta_0 \w \bar b = 0
\]
for all $b \in H^0(K)$.  Equivalently, 
\[
\int_\Sigma \theta_0 \w a = 0 \text{ for all harmonic $(0,1)$ forms $a$}.
\]

\subsubsection{Relationship to the $S^1$-action on the moduli space of all parabolic Higgs bundles}
The moduli space of all parabolic Higgs bundles (where the eigenvalues of residues are not fixed) is a holomorphic Poisson manifold \cite{LogMart} whose symplectic leaves are the hyperk\"ahler manifolds defined above.  From our perspective, we let $\P$ be the moduli space of solutions to Hitchin's equations (\ref{Hitchinseqs}) as above except we no longer fix the eigenvalues of the residues of the Higgs field.  There is a natural $S^1$-action on $\P$ given by scaling the Higgs field and we let $\tilde X$ denote the vector field on $\P$ that generates it.  We will show that there is a canonical projection map $p :T\P \to T\M$, where $T\M \subset T\P$ is the distribution underlying the foliation by Higgs fields whose residues have fixed eigenvalues, and that $p(\tilde X)$ is the vector field $X$ of theorem \ref{thm:mainX} (i.e. $\alpha = i_{p(\tilde X)} \om_I$).

By linearizing (\ref{Hitchinseqs}), at a point $(A,\theta) \in \P$ we have
\begin{align*}
T_{(A,\theta)} \P &= \frac{\left\{(a,b) \mid \begin{aligned}
\partial_E a - \bar\partial_E a^* + [\theta,b^*] + [b,\theta^*] = 0, \\
 [a,\theta] + \bar\partial_E b = 0.
 \end{aligned} \right\}}{\{(\bar\partial_A Y, [\theta,Y]) \mid Y \in \A^0(\operatorname{Par}\su(E))\}} \\
&\subset \frac{\A^{0,1}(\sl(E)) \times \A^{1,0}(\operatorname{Par}\mathfrak{sl}(E)(D))}{\A^0(\operatorname{Par}\mathfrak{su}(E))}
\end{align*}
and similarly for $T_{(A,\theta)}\M$, except that $b$ must lie in $\A^{1,0}(\operatorname{SPar}\mathfrak{sl}(E)(D))$.  It is then straightforward to check that
\[
\tilde \om_I : T\P \otimes T\M \to \R, ~~ \tilde \om_I((a_1,b_1),(a_2,b_2)) = \int_\Sigma \tr(a_1^*\w a_2 - a_2^* \w a_1 - b_1 \w b_2^* + b_2 \w b_1^*)
\]
is well-defined and restricts to the K\"ahler form $\om_I$ on any symplectic leaf.  This in turn defines a natural projection $p : T\P \to T\M$ via $\om_I(p(v),w) = \tilde \om_I(v, w)$ for all $w \in T\M$.  The vector field $\tilde X$ on $\P$ is given by $\tilde X_{(A,\theta)} = (0,i\theta)$ and $X$ is defined by $i_X \om_I = \alpha$. From (\ref{alphaHiggs}) and the definition of $\tilde \om_I$, we therefore see that $p(\tilde X) = X$.

\subsection{Nakajima quiver varieties}
We first recall the following general set-up.  Suppose $V$ is a hermitian vector space.  Then $T^* V \simeq V\times V^*$ has the structure of a hyperk\"ahler manifold, where $I$ is given by multiplication by $i$ and 
\[
J(v',w') = (-w'^*, v'^*), ~~ (v',w') \in V\oplus V^* \simeq T_{(v,w)} (T^* V),
\]
where $*: V \longleftrightarrow V^*$ is the conjugate linear isomorphism determined by the hermitian metric.  The action of $S^1$ on $T^* V$ given by scalar multiplication on $V^*$ satisfies the hypotheses of theorem \ref{thm:11form} and the moment map of this action is
\[
\mu : T^* V \to \R, ~~ (v,w) \mapsto -\frac{1}{2} ||w||^2.
\]
Letting $\alpha = d\mu$, we record for later use the following useful equations:
\begin{gather}
\alpha_{(v,w)} (v',w') = -g((0,w), (0,w')) \label{vspacealpha} \\
(J\alpha + i K\alpha)_{(v, w)}(v',w') = (\om_J + i\om_K)((0,w), (v',w')). \label{10alpha} 
\end{gather}

We will now consider the Nakajima quiver varieties \cite{Nakajima:1994}. Let $\mathcal Q$ be a quiver (i.e. directed graph) with $n$ vertices labeled $\{1,2,\ldots,n\}$ and let $\mathbf v = (v_1,\ldots, v_n), \mathbf w = (w_1,\ldots, w_n) \in \Z^n$.  Let
\[
M_{\mathcal Q} = \bigoplus_{i\to j \in \mathcal Q} \Hom(\C^{v_i}, \C^{v_j}) \oplus \bigoplus_{k=1}^n \Hom(\C^{w_k}, \C^{v_k})
\]
be the space of representations of the framed quiver of $\mathcal Q$ (here we are identifying $\mathcal Q$ with its set of oriented edges).  Then $M_{\mathcal Q}$ is a hermitian vector space (coming from the standard hermitian structure on $\C^m$) so that, by the preceding discussion, $T^* M_{\mathcal Q}$ is a hyperk\"ahler vector space with $S^1$-action satisfying the conditions of theorem \ref{thm:11form}.  If $H$ denotes the union of the edges of $\mathcal Q$ along with the edges with the opposite orientation then we have
\[
T^*M_{\mathcal Q} = \bigoplus_{i\to j \in H} \Hom(\C^{v_i}, \C^{v_j}) \oplus \bigoplus_{k=1}^n \Hom(\C^{w_k}, \C^{v_k}) \oplus \bigoplus_{k=1}^n \Hom(\C^{v_k}, \C^{w_k}).
\]
Using the notation of \cite{Nakajima:1994}, we write $(B, i, j) = (B_h, i_k, j_k)_{h\in H,k\in\{1,\ldots,n\}}$ for an element of $T^*M_{\mathcal Q}$, where $h\in H$ is an arrow from vertex $s(h)$ to $t(h)$, $B_h \in \Hom(\C^{v_{s(h)}}, \C^{v_{t(h)}}), i_k \in \Hom(\C^{w_k}, \C^{v_k}),$ and $j_k \in \Hom(\C^{v_k}, \C^{w_k})$.  Then we have
\begin{align}
(\om_J + i \om_K)((B, i, j), (B',i',j')) = \sum_{h\in H} & \tr(\eps(h) B_h B_{h'}) \nonumber\\
& + \sum_{k=1}^n \tr(i_k j_k' - i_k' j_k), \label{holsym}
\end{align}
where we are identifying a fiber of $T(T^* M_{\mathcal Q})$ with $T^* M_{\mathcal Q}$ and $\eps(h) = 1$ if $h \in \mathcal Q$ and $-1$ otherwise.

The group
\[
G_{\mathbf v} = \prod_{j=1}^n U(\C^{v_j})
\]
acts on $T^* M_\Om$ preserving the hyperk\"ahler structure.  The action vector field corresponding to $Y = (Y_1,\ldots, Y_n) \in \g_{\mathbf v}$ is 
\[
Y_{(B,i,j)} = (Y_{t(h)} B_h - B_h Y_{s(h)}, Y_k i_k, -j_k Y_k)_{h \in H, k\in\{1,\ldots,n\}} 
\]
Let $Z_{\mathbf v} \simeq \u(1)^{\oplus n}$ denote the center of the Lie algebra $\g_{\mathbf v}$ of $G_{\mathbf v}$ and fix elements $\zeta_\R = ((\zeta_\R)_1, \ldots, (\zeta_\R)_n)\in Z_{\mathbf v}$ and $\zeta_\C = ((\zeta_\C)_1, \ldots, (\zeta_\C)_n) \in Z_{\mathbf v}\otimes\C$.  Then, identifying $\g_{\mathbf v}$ with its dual via the hermitian inner product, hyperk\"ahler moment maps are
\begin{gather*}
\mu_I : T^* M_{\mathcal Q} \to \g_v = \bigoplus_{k=1}^n \u(\C^{v_k}) \\
\mu_J + i \mu_K : T^* M_{\mathcal Q} \to \g_v\otimes\C = \bigoplus_{k=1}^n \gl(\C^{v_k}) 
\end{gather*}
whose $k$th components are
\begin{gather*}
(\mu_I(B,i,j))_k = \frac{1}{2} \left(\sum_{h \in H, t(h) = k} (B_h B_h^* - B_{\bar h}^* B_{\bar h} + i_k i_k^* - j_k^* j_k) \right) - (\zeta_\R)_k \in \u(\C^{v_k})\\
((\mu_J + i \mu_K)(B,i,j))_k = \sum_{h\in H, t(h) = k} (\eps(h) B_h B_{\bar h} + i_k j_k) - (\zeta_\C)_k \in \gl(\C^{v_k}),
\end{gather*}
where $\bar h$ denotes the edge $h$ with the opposite orientation.  We let $\M_{(\zeta_\R, \zeta_\C)}$ be the hyperk\"ahler quotient with respect to this moment map.

From the discussion at the beginning of this section, $\alpha$ satisfies
\[
d\alpha = 0, ~~ d(J\alpha) = \om_J, ~~ d(K\alpha) = \om_K
\]
and from equations (\ref{vspacealpha}) and (\ref{10alpha}) along with (\ref{holsym}), it is straightforward to see that
\begin{gather*}
\alpha(Y^*) = 0 \\
(J\alpha + i K\alpha)(Y^*) = -2 \sum_{j=1}^n (\zeta_\C)_j \tr(Y_j),
\end{gather*}
where the last equation is restricted to the 0 level set of the moment map $(\mu_I, \mu_J, \mu_K)$.  Thus by proposition \ref{prop:reduction}, $\M_{(\zeta_\R, \zeta_\C)}$ has a hyperholomorphic Lie algebroid.  The above equation along with corollary \ref{quantization} gives
\begin{proposition}
The symplectic form $\om_J$ on $\M_{(\zeta_\R, \zeta_\C)}$ is pre-quantizable if $\Im (\zeta_\C)_1, \ldots, \Im(\zeta_\C)_n \in \frac{1}{2}\Z$.  Similarly, $\om_K$ is pre-quantizable if \\
$\Re (\zeta_\C)_1, \ldots, \Re(\zeta_\C)_n \in \frac{1}{2}\Z$.
\end{proposition}

\subsection{Nahm's equations}
Now we will show that the moduli space of solutions to Nahm's equations \cite{Biel:2007,Kron:1990} has a canonical 1-form $\alpha$ satisfying (\ref{alphaprops}).  Let $G$ be a compact Lie group with an $\Ad$-invariant inner product $\langle \cdot, \cdot\rangle$ on its Lie algebra $\g$.  Fix $\tau_1,\tau_2,\tau_3 \in \g$ such that the intersection of the centralizers is a Cartan subalgebra $\h \subset \g$.  We define
\[
\mathcal A_{\tau_1,\tau_2,\tau_3} = \{ T_0 + i T_1 + j T_2 + k T_3 : [0,\infty) \to \g \otimes \mathbb H \mid \substack{T_0 \to 0, \\ T_i \to \tau_i, i=1,\cdots 3}\},
\]
where the convergence is exponentially fast.  Write $\underline T$ for $(T_0,T_1,T_2,T_3)$.

The space $\mathcal A_{\tau_1,\tau_2,\tau_3}$ is an affine space modeled on $\mathcal A_{0,0,0}$ and we have a hyperk\"ahler structure defined as follows.  The metric is given by
\begin{gather*}
||(t_0, t_1, t_2, t_3)||^2 = \int_0^\infty \sum_{j=0}^3 \langle t_j(s), t_j(s)\rangle ds, \\
 (t_0,t_1,t_2, t_3) \in \A_{0,0,0} \simeq T_{\underline T} \A_{\tau_1,\tau_2,\tau_3}
\end{gather*}
and the complex structures $I, J, K$ are given by right multiplication by $-i, -j, -k$, respectively.

Let
\[
\G = \{ g : [0,\infty) \to G \mid g(0) = e, \lim_{s \to \infty} g(s) \in \exp\h\}
\]
where $e \in G$ is the identity element.  Then $\G$ acts on $\mathcal A_{\tau_1,\tau_2,\tau_3}$ via
\[
g \cdot (T_0,T_1,T_2,T_3) = (\Ad_g T_0 - \dot g g^{-1}, \Ad_g T_1, \Ad_g T_2, \Ad_g T_3),
\]
preserving the hyperk\"ahler structure.  The action is Hamiltonian and we have the moment maps
\begin{gather*}
\mu_I(T_0,T_1,T_2,T_3) = \dot T_1 + [T_0, T_1]- [T_2, T_3] \\
\mu_J(T_0,T_1,T_2,T_3) = \dot T_2 + [T_0, T_2]- [T_3, T_1] \\
\mu_K(T_0,T_1,T_2,T_3) = \dot T_3 + [T_0, T_3]- [T_1, T_2],
\end{gather*}
which are called Nahm's equations.  Let $\M = \mathcal A_{\tau_1,\tau_2,\tau_3} // \G$ be the hyperk\"ahler reduction at 0.  

Now suppose $\tau_2 = \tau_3 = 0$.  Then we have an $S^1$ action on $\A_{\tau_1,0,0}$ via
\[
e^{i\theta} \cdot (T_0,T_1,T_2 + i T_3) = (T_0, T_1, e^{i\theta}(T_2+iT_3)),
\]
which commutes with the $\G$-action.  The moment map is given by $\underline T \mapsto -\int_0^\infty \langle T_2(s), T_2(s) \rangle + \langle T_3(s), T_3(s) \rangle$ and its exterior derivative is the $\G$-invariant 1-form
\[
\alpha_{\underline T}(\underline t) = -\int_0^\infty (\langle T_2(s), t_2(s)\rangle + \langle T_3(s), t_3(s)\rangle) ds.
\]
If $\tau_2,\tau_3 \ne 0$ then the integral defining the moment map diverges but $\alpha$ is still well-defined since $t_2$ and $t_3$ converge to 0 exponentially fast and $T_2$ and $T_3$ are bounded.

The action vector field corresponding to
\[
Y \in \Lie(\G) = \{Y : [0, \infty) \to \g \mid Y(0) = 0, \lim_{s\to\infty} Y(s) \in \h\}
\]
is given by
\begin{equation}\label{actionfield}
Y^*_{(T_0, T_1, T_2, T_3)} = ([Y, T_0] - \dot Y, [Y, T_1], [Y, T_2], [Y, T_3]).
\end{equation}

It is straightforward to verify that on $\A_{\tau_1,\tau_2,\tau_3}$, $\alpha$ satisfies
\[
d\alpha = 0, ~ d(J\alpha) = \om_J, ~ d(K\alpha) = \om_K, ~ \alpha(Y^*) = 0.
\]
To invoke proposition \ref{prop:reduction}, we just need to show that the (necessarily locally constant) functions $\underline T\mapsto (J\alpha)_{\underline T}(Y^*), (K\alpha)_{\underline T}(Y^*)$ are indeed constant on the level set of the moment map.
\begin{proposition}
On the solution space to Nahm's equations, we have
\[
(J\alpha)(Y^*) = -\langle \tau_2, Y(\infty) \rangle, ~~ (K\alpha)(Y^*) = -\langle \tau_3, Y(\infty)\rangle.
\]
\end{proposition}
\begin{proof}
We have
\[
(J\alpha)_{\underline T}(\underline t) = \alpha_{\underline T}(t_2, t_3, -t_0, -t_1) = \int_0^\infty \langle T_2, t_0\rangle + \langle T_3, t_1\rangle
\]
so that using (\ref{actionfield}) we compute
\begin{align*}
(J\alpha)_{\underline T}(Y^*) &= \int_0^\infty \left(\langle T_2, [Y, T_0] - \dot Y\rangle + \langle T_3, [X, T_1]\rangle\right) \\
&= \int_0^\infty \left(-\frac{d}{ds} \langle T_2, Y\rangle + \langle \dot T_2, Y\rangle + \langle T_2, [Y, T_0]\rangle + \langle T_3, [X,T_1]\rangle \hspace{-4pt} \right) \\
&= -\langle \tau_2, Y(\infty) \rangle + \int_0^\infty \langle \underbrace{\dot T_2 + [T_0, T_2] + [T_1, T_3]}_{\text {Nahm's eq}}, X\rangle \\
&= -\langle \tau_2, Y(\infty)\rangle.
\end{align*}

The proof for $K\alpha$ is similar.
\end{proof}

From corollary \ref{quantization} we have
\begin{proposition}
The symplectic form $\om_J$ (resp. $\om_K$) is pre-quantizable if the element in $\h^*$ dual to $\tau_2$ (resp. $\tau_3$) via the Killing form lies in the weight lattice.
\end{proposition}

The space $\M$ is diffeomorphic to a complex coadjoint orbit of the complexification of $G$ \cite{Kron:1990} and is therefore simply-connected if $G$ is.  Thus in this case the construction of section \ref{sec:twistor} can be used to construct the holomorphic line bundle on its twistor space.

\appendix
\section{Construction of the line bundle/Lie algebroid on twistor space}
\subsection{Definition of $\phi_j$}

Let $F = F_1 + i F_2$ and
\[
\tilde F = \frac{1}{\zeta} F - \zeta \overline F \in \A^{1,1}_V(2).
\]
Let $Y = X + i\zeta\frac{\partial}{\partial \zeta}$ where $X$ is the vector field for the infinitesimal $S^1$-action and the sum is via the $C^\infty$ decomposition $TZ = T M \oplus T\C P^1$. 

Using (\ref{Xprops}), it is straightforward to verify that
\begin{equation}
\L_Y \om = \tilde F. \label{liederiv}
\end{equation}

Let $\{U_j\}$ be an open cover of $U$ such that on each $U_j$ we have:
\begin{itemize}
\item A holomorphic lift $Z_j \in \Gamma(U_j; T_Z^{1,0})$ of $\frac{\partial}{\partial \zeta}$ satisfying
\begin{equation}
\L_{Z_j} i\zeta \om = 0, \label{darb}
\end{equation}
which is equivalent to  
\begin{equation}
\L_{Z_j} \om = -\frac{1}{\zeta} \om. \label{liederiv2}
\end{equation}

\item A 1-form $A_j \in \Gamma(U_j; \Lambda^{1,0} T_V^*)$ such that $dA_j = \overline F\vert_{U_j}$\footnote{Such an $A_j$ exists since we can find (at least locally) a line bundle and connection with curvature $\bar F$.  Then  the (0,1) part defines a holomorphic structure so we can take $A_j$ to be the connection forms in a holomorphic gauge.}.
\end{itemize}

Each $Z_j$ defines a local holomorphic splitting $TZ \simeq TM \oplus T\C P^1$ over $U_j$.  If $\mu \in \A_v^\bullet$ then we let $\mu_j \in \A^\bullet_{U_j}$ denote the corresponding differential form with respect to this splitting.  

Define the vertical vector fields
\[
X_{jk} = Z_k - Z_j \in \Gamma(T_V\vert_{U_j\cap U_k })
\]
and
\[
T_j = Z_j - \frac{\partial}{\partial \zeta} \in \Gamma(T_V\vert_{U_j}).
\]

We will make use of the following.
\begin{lemma} \label{lemma:split}
For any $\mu \in \Gamma(Z; \Lambda^\bullet T_V^*)$, we have
\[
\mu_k - \mu_j = -d\zeta \w i_{X_{jk}} \mu
\]
\[
d \mu_j = (d_V \mu)_j + d\zeta \w (\L_{Z_j} \mu)_j
\]
and if $v \in \Gamma(TZ)$ projects to a vector field on $\C P^1$, then
\[
\L_v \mu_j = (\L_v \mu)_j + d\zeta \w (i_{[Z_j, v_j]} \mu)_j,
\]
where $v_j$ is the vertical part of $v$ with respect to the splitting determined by $Z_j$.
\end{lemma}
\begin{proof}
All of the equations are easily seen to be true when pulled back to $\Lambda^\bullet T_V^*$, so it is sufficient to check that the equations are true when contracted with $Z_j$ and then pulled back to $\Lambda^{\bullet-1} T_V^*$.  Let $\iota: T_V \to T_Z$ be the inclusion.  We have
\[
\iota^* i_{Z_j} (\mu_k - \mu_j) = \iota^* i_{Z_j} \mu_k = - \iota^* i_{X_{jk}} \mu_k = - i_{X_{jk}} \mu = i_{Z_j} (d\zeta \w i_{X_{jk}} \mu),
\]
proving the first equation.  The second equation follows from
\[
\iota^* i_{Z_j} d \mu_j = \iota^*\left(- d i_{Z_j} \mu_j + \L_{Z_j} \mu_j\right) = \iota^* \L_{Z_j} \mu_j = \L_{Z_j} \mu
\]
and the third from
\[
\iota^* i_{Z_j} \L_v \mu_j = \iota^*\left(\L_v i_{Z_j} \mu_j + [i_{Z_j}, \L_v] \mu_j \right) = \iota^* i_{[Z_j, v]} \mu_j = i_{[Z_j, v_j]} \mu.
\]
\end{proof}

\begin{proposition}
The 1-form
\[
i_{[Y, i\zeta Z_j]} \om + i_{i\zeta T_j} \tilde F - 2 i \zeta A_j \in \A^1_V(U_j)
\]
is (vertically) closed.
\end{proposition}
\begin{proof}
Using (\ref{liederiv}) and (\ref{liederiv2}), we have
\begin{gather*}
d_V i_{[Y, i\zeta Z_j]} \om = \L_{[Y, i\zeta Z_j]} \om = [\L_Y, \L_{i\zeta Z_j}] \om = -i \tilde F - \L_{i\zeta Z_j} \tilde F \\
= -i \tilde F - \L_{i\zeta \frac{\partial}{\partial \zeta}} \tilde F - \L_{i\zeta T_j} \tilde F = -i \tilde F + \frac{i}{\zeta} F + i \zeta \bar F - d_V i_{i\zeta T_j} \tilde F \\
= 2 i \zeta \bar F - d_V i_{i\zeta T_j} \tilde F = d_V (2i \zeta A_j) - d_V i_{i\zeta T_j} \tilde F.
\end{gather*}
\end{proof}

Therefore, shrinking $U_j$ if necessary, we can find $f_j \in C^\infty(U_j)$ such that
\begin{equation}
d_V f_j = i_{[Y, i\zeta Z_j]} \om + i_{i\zeta T_j} \tilde F - 2 i \zeta A_j \in \A^1_V(U_j) \label{def:f}
\end{equation}

\begin{proposition}
The function $i_X i_{X_{jk}} i\zeta \om - f_k + f_j$ is constant on $Z_0$. 
\end{proposition}
\begin{proof}
On $Z_0$ we have
\begin{align*}
d(f_k - f_j) &= i_{[X+i\zeta\frac{\partial}{\partial\zeta}, i\zeta X_{jk}]} \om + i_{i\zeta X_{jk}}\tilde F \\
&=  i_{[X, X_{jk}]} i\zeta \om - i_{X_{jk}} \zeta \om + i_{i\zeta X_{jk}}\tilde F \\
&= i_{[X, X_{jk}]} i\zeta \om + i i_{X_{jk}} (\om_J + i \om_K) + i_{i\zeta X_{jk}}\tilde F
\end{align*}
and
\begin{align*}
d i_X i_{X_{jk}} i \zeta \om &= \L_X i_{X_{jk}} i \zeta \om - i_X d i_{X_{jk}} i \zeta \om \\
&= i_{X_{jk}} \L_X i \zeta \om + i_{[X, X_{jk}}] i \zeta \om - i_X \L_{X_{jk}} i \zeta \om \\
&= i_{X_{jk}} \L_X i \zeta \om + i_{[X, X_{jk}}] i \zeta \om \\
&= i\zeta i_{X_{jk}} \L_{Y- i \zeta \frac{\partial}{\partial \zeta}} \om + i_{[X, X_{jk}]} i \zeta \om \\
&= i \zeta i_{X_{jk}} \tilde F + \zeta^2 i_{X_{jk}} \L_{\frac{\partial}{\partial \zeta}}\left(\frac{1}{i\zeta}(\om_J + i \om_K)\right) +  i_{[X, X_{jk}]} i \zeta \om \\
&= i \zeta i_{X_{jk}} \tilde F + i i_{X_{jk}} (\om_J + i \om_K) + i_{[X, X_{jk}]} i \zeta \om.
\end{align*}

\end{proof}
Thus $i_X i_{X_{jk}} i\zeta \om - f_k + f_j$ is a C\v ech 1-cocycle with values in the constant sheaf $\C$.  By our assumption that $\pi_1 M = 0$, we have
\begin{corollary}\label{cor:vanish}
We can choose the $f_j$ so that the function $i_X i_{X_{jk}} i\zeta \om - f_k + f_j$ vanishes on $Z_0$.
\end{corollary}

We are now able to define the 1-forms $\phi_j$ that will define the Lie algebroid.
\begin{definition}
Let
\[
\phi_j = i_Y \om_j - \frac{f_j}{i\zeta} d\zeta + 2 \zeta A_j \in \A^{1,0}_Z(U_j).
\]
\end{definition}

\begin{proposition}\label{prop:phij}
The form $\phi_k - \phi_j$ is non-singular at $Z_0$ and we have
\[
d\phi_j = i_Y \left( d\zeta \w \frac{1}{\zeta} \om \right) + \frac{1}{\zeta} F + \zeta \overline F.
\]
\end{proposition}
\begin{proof}
From lemma \ref{lemma:split} and corollary \ref{cor:vanish}, at $Z_0$ we have
\begin{align*}
\phi_k - \phi_j &= -i_Y \left(d\zeta \w i_{X_{jk}} \om\right) - i_X i _{X_{jk}} \om d\zeta \\
&= -i \zeta i_{X_{jk}} \om + d\zeta \w i_X i_{X_{jk}} \om - i_X i _{X_{jk}} \om d\zeta \\
&= -i \zeta i_{X_{jk}} \om.
\end{align*}
is non-singular.  

We use lemma \ref{lemma:split} and (\ref{liederiv2}) to compute
\begin{align*}
d(i_Y \om_j) &= \L_Y \om_j - i_Y d\om_J \\
&= (\L_Y \om)_j + d\zeta \w i_{[Z_j, Y_j]} \om - i_Y\left(d\zeta \w \L_{Z_j} \om\right) \\
&= (\tilde F)_j + d\zeta \w i_{[Z_j, Y_j]} \om + i_Y\left(d\zeta \w \frac{1}{\zeta} \om\right)
\end{align*}
while
\begin{align*}
d\left(\frac{-f_j}{i\zeta} d\zeta + 2\zeta A_j\right) &=  -\left(\frac{1}{i\zeta} i_{[Y, i \zeta Z_j]} \om + i_{T_j} \tilde F - 2 \iota^* A_j\right)\w d\zeta + 2 d\zeta \w A_j + 2 \zeta \bar F \\
&= d\zeta \w i_{[Y_j, Z_j]} +d\zeta \w i_{T_j} \tilde F + 2 \zeta \bar F.
\end{align*}
Since $\tilde F = (\tilde F)_j + d\zeta \w i_{Z_j} \tilde F$, summing these gives
\[
d\phi_j = i_Y\left(d\zeta \w \frac{1}{\zeta} \om\right) + \tilde F + 2\zeta \bar F = i_Y\left(d\zeta \w \frac{1}{\zeta} \om\right) + \frac{1}{\zeta} F + \zeta \bar F
\]
as desired.

\end{proof}

\begin{corollary}
The forms $\phi_k - \phi_j$ give a cocycle of holomorphic 1-forms on $U$ and therefore defines a holomorphic Lie algebroid on $U$.
\end{corollary}

It is straightforward to check that $Y$ and $\om$ are invariant under $\overline \tau^*$ while $d\zeta/\zeta$ and $\frac{1}{\zeta} F + \zeta \overline F$ each pick up a factor of $-1$.  Thus from the above proposition, we have $\overline{\tau^*}(d\phi_j) = - d\phi_j$.  Therefore we can extend this Lie algebroid to all of $Z$ by using the cocycle  $\{\phi_k - \phi_j, \overline{\tau^*}(\phi_k) - \phi_j, \overline{\tau^*}(\phi_j) - \overline{\tau^*}(\phi_k) \}$ relative to the cover $\{U_j\} \cup \{\tau(U_j)\}$ of $Z$.  Then
$\{\phi_j, -\overline{\tau^*}(\phi_k)\}$ gives a connection with singularities at $\zeta = 0, \zeta = \infty$ of curvature $i_Y \left( d\zeta \w \frac{1}{\zeta} \om \right) + \frac{1}{\zeta} F + \zeta \overline F$.

\subsection{Computation of the Atiyah class}
We now show that the Lie algebroid constructed in the previous section is the same one determined by the (1,1)-form $2 i \om_I - 2 i d(I\alpha)$.  Any holomorphic Lie algebroid on $Z$ has an Atiyah class in $H^1(\Om^1_Z) = H^{1,1}(Z)$ coming from the long exact sequence in cohomology associated to the short exact sequence of sheaves
\[
0 \to d\O_Z \to \Om^1_Z \to d \Om^1_Z \to 0.
\]
Since $H^0(d\Om_Z^1) = 0$ for twistor space \cite{Hitchin:2012}, the map $H^1(d\O_Z) \to H^1(\Om^1_Z) \simeq H^{1,1}(Z)$ is injective so that a holomorphic Lie algebroid is completely determined by its Atiyah class in $H^{1,1}(Z)$.  We will prove
\begin{theorem}\label{thm:chernclass}
The Atiyah class of the holomorphic Lie algebroid defined in the previous section is
\[
2 i \om_I - 2 i d(I\alpha) \in H^{1,1}(Z).
\]
\end{theorem}

To compute this characteristic class, we find a (1,0) form $A - \frac{\mu}{\zeta} d\zeta$ whose residues at $\zeta = 0$ and $\zeta = \infty$ agree with those of $\phi_j$ and $\phi_{\tilde j}$, respectively.  It then follows that the Atiyah class of this Lie algebroid in $H^{1,1}(Z)$ is given by
\[
\bar\partial \phi_j - \bar\partial\left(A - \frac{\mu}{\zeta} d\zeta\right).
\]

It is straightforward to verify
\begin{equation}
i_X \om_J = -K\alpha, ~~ i_X \om_K = J\alpha \label{interior},
\end{equation}
from which we see that the singular part of $\phi_j$ at $\zeta = 0$ is
\[
(\zeta\phi_j)\vert_{\zeta = 0} = (J\alpha + iK\alpha)_j + i f_j d\zeta.
\]

If $\gamma$ is a vertical differential form, we will abuse notation and write $\gamma$ for the corresponding differential form on $Z$ obtained via the global $C^\infty$ splitting $TZ = TM \times T\C P^1$.  Then for any $\gamma \in \A^1(M)$ we have
\[
\gamma - \gamma_j = \gamma(T_j) d\zeta.
\]

Therefore in terms of the $C^\infty$ splitting we have

\begin{equation}
(\zeta\phi_j)\vert_{\zeta = 0} = J\alpha + i K\alpha + (if_j\vert_{Z_0} - (J\alpha + i K\alpha)(T_j)) d\zeta. \label{resCinfty}
\end{equation}

\begin{proposition}
On $Z_0 \simeq M$, the functions $if_j - (J\alpha + i K\alpha)(T_j)$ piece together to a global function $\mu$ satisfying
\[
d\mu = 2\alpha.
\]
\end{proposition}
\begin{proof}
From corollary \ref{cor:vanish}, the difference of these functions is
\[
i i_X i_{X_{jk}} i \zeta \om - (J\alpha + i K\alpha)(X_{jk}).
\]

By (\ref{interior}), the above becomes
\[
i i_X i_{X_{jk}} i \zeta \om + i i_{X_{jk}} i_X (\om_J + i \om_K) = 0
\]
on $Z_0$.

Now, up to first order in $\zeta$, $[Y, i\zeta Z_j] = i\zeta[X,Z_j] - \zeta T_j$.  So from (\ref{def:f}) on $Z_0$ we have
\begin{align*}
d f_j &= i_{[X,Z_j]} (\om_J + i \om_K) + i i_{T_j} (\om_J + i\om_K) + i i_{T_j} F \\
&=  i_{[X,T_j]} (\om_J + i \om_K) + i i_{T_j} (\om_J + i\om_K + F) \\
&= i_{[X,T_j]} (\om_J + i \om_K) + i i_{T_j} d(J\alpha + i K\alpha) \\
&= i_{[X,T_j]} (\om_J + i \om_K) - i d i_{T_j} d(J\alpha + i K\alpha) + i \L_{T_j}(J\alpha + iK\alpha) \\
&= i_{[X,T_j]} (\om_J + i \om_K) - i d i_{T_j} d(J\alpha + i K\alpha) + \L_{T_j} i_X (\om_J + \om_K) \\
&= (i_{[X,T_j]} + \L_{T_j} \circ i_X)(\om_J + i \om_K) - i d i_{T_j} d(J\alpha + i K\alpha) \\
&= i_X \L_{T_j} (\om_J + i \om_K) - i d i_{T_j} d(J\alpha + i K\alpha).
\end{align*}
From the degree zero part (in $\zeta$) of (\ref{darb}), one sees that $\L_{Z_j}(\om_J + i\om_K) = -2 i \om_I$ so that
\begin{gather*}
d\mu = i i_X \L_{T_j}(\om_J + i \om_K) = i i_X \L_{Z_j}(\om_J + i \om_K) = 2 i_X \om_I = 2\alpha.
\end{gather*}

\end{proof}

Now let
\[
A = \frac{1}{\zeta} (J\alpha + i K\alpha) + 2i I\alpha + \zeta(J\alpha - iK\alpha),
\]
which is a 1-form on $Z$ of type (1,0).  Then from (\ref{resCinfty}) we see that
\[
\left(\zeta\left(A + \frac{\mu}{\zeta} d\zeta\right)\right)_{\zeta = 0} = (\zeta\phi_j)\vert_{\zeta = 0}.
\]
so that $\phi_j - A - \frac{\mu}{\zeta} d\zeta$ is non-singular at $\zeta = 0$.

Since $\overline{\tau^*}(A + \frac{\mu}{\zeta}d\zeta) = -A - \frac{\mu}{\zeta} d\zeta$, we see that $\{\phi_j - A - \frac{\mu}{\zeta} d\zeta, -\overline{\tau^*}(\phi_j) - A - \frac{\mu}{\zeta} d\zeta\}$ is a non-singular connection for the Lie algebroid constructed in the previous section.  Thus the Atiyah class is represented by $\bar\partial(\phi_j - A - \frac{\mu}{\zeta} d\zeta)$.

From proposition \ref{prop:phij} we have
\[
\bar\partial \phi_j = \frac{1}{\zeta} F + \zeta \bar F.
\]
Now we have
\begin{align*}
dA = \frac{1}{\zeta}(\om_J + i \om_K + F) + 2i d(I\alpha) + \zeta(\om_J - i \om_K + \overline F) \\
- \frac{1}{\zeta} d\zeta \w \left(\frac{1}{\zeta} (J\alpha + i K\alpha) - \zeta(J\alpha - i K\alpha)\right).
\end{align*}
Now, since $\frac{1}{\zeta}(\om_J + i\om_K) + 2 i \om_I + \zeta(\om_J - i\om_K)$ is of type (2,0), we have
\[
\left(\frac{1}{\zeta}(\om_J + i\om_K) + 2 i \om_I + \zeta(\om_J - i\om_K)\right)_{1,1} = -(2i\om_I)_{1,1}.
\]
Therefore, since $2 i d(I\alpha) - 2 i \om_I$ is of type (1,1) we see that
\[
\bar\partial\left(\phi_j - A - \frac{\mu}{\zeta} d\zeta\right) = 2 i \om_I  - 2 i d(I\alpha) + \frac{1}{\zeta} d\zeta \w \hspace{-2pt} \left(\frac{1}{\zeta} (J\alpha + i K\alpha) - \zeta(J\alpha - i K\alpha) + 2\alpha \right)_{0,1}. 
\]
But this last term vanishes since 
\begin{align*}
\frac{1}{\zeta}(J\alpha + & i K\alpha) + 2\alpha - \zeta(J\alpha - i K\alpha) \\
&= i\left( \frac{1}{\zeta} (J(I\alpha) + i K(I\alpha)) + 2 i I(I\alpha) + \zeta(J(I\alpha) - i K(I\alpha) \right)
\end{align*}
is of type (1,0).  This proves theorem \ref{thm:chernclass}.

\bibliography{biblio}
\bibliographystyle{plain}

\end{document}